\documentclass[11pt,reqno]{amsart}

\usepackage{amsfonts,amssymb,latexsym,amsmath}
\usepackage{xcolor}
\usepackage{amssymb}
\usepackage{tikz}
\usepackage{tabularx}

\newtheorem{theorem}{Theorem}[section]

\newtheorem{corollary}[theorem]{Corollary}
\newtheorem{proposition}[theorem]{Proposition}

\theoremstyle{definition}

\numberwithin{equation}{section}

\DeclareMathOperator{\tr}{trace}

\newcommand{\C}{{\mathbb C}}

\newcommand{\T}{{\mathbb T}}
\renewcommand{\kappa}{\varkappa}
\newcommand{\be}{\begin{equation}}
\newcommand{\ee}{\end{equation}}

\newcommand{\bq}{\begin{eqnarray}}
\newcommand{\eq}{\end{eqnarray}}
\newcommand{\nn}{\nonumber}
\newcommand{\ba}{\begin{array}}
\newcommand{\ea}{\end{array}}

\newcommand{\iv}{^{-1}}
\newcommand{\iy}{\infty}

\newcommand{\cL}{\mathcal{L}}

\newcommand{\cS}{\mathcal{S}}

\newcommand{\B}{\mathcal{B}}

\newcommand{\cC}{\mathcal{C}}

\begin{document}

\title[Asymptotics of determinants for structured matrices]{A survey of asymptotics of determinants for structured matrices}

\author[E. Basor]{Estelle Basor}
\address{American Institute of Mathematics, Pasadena, CA-91125, USA
}
\email{ebasor@aimath.org}

\author[T. Ehrhardt]{Torsten Ehrhardt     }
\address{Department of Mathematics, University of California, Santa Cruz, CA-95064,  USA}
\email{tehrhard@ucsc.edu}

\author[J. Virtanen]{Jani Virtanen}
\address{Department of Mathematics, University of Reading, England}
\email{j.a.virtanen@reading.ac.uk}
\thanks{The third author was supported in part by EPSRC grant EP/Y008375/1.}

\begin{abstract} 
 In this survey we show how to produce asymptotics of determinants of structured matrices using operator theory methods. We describe the asymptotics for  finite Toeplitz matrices, finite Toeplitz plus Hankel matrices and generalizations of these, and also finite sections of functions of Toeplitz operators, all with smooth matrix-valued symbols. Many of these results are well-known. However, the results for certain Toeplitz plus Hankel operators with matrix-valued symbol are presented for the first time and the result for the finite sections of functions of Toeplitz operators is new.
 \end{abstract}

\maketitle

\section{introduction}
The goal of this survey is to show how using operator theory along with a fact about triangular matrices can produce asymptotics for determinants of structured matrices. We will describe results for finite Toeplitz matrices, finite Toeplitz plus Hankel matrices and also finite sections of functions of Toeplitz operators, all with smooth matrix-valued symbols. The section on Toeplitz plus Hankel matrices for matrix-valued symbols is implicit in previous work \cite{BE07}, but stated here for the first time. The result for the finite sections of functions of Toeplitz operators is new. Each section of the paper will deal with one of the three above mentioned cases. 

This survey is not meant to describe results with the best possible conditions or assumptions, but rather to show how easily one can produce the results using operator theory. There is a long history of using these methods. The first use was in the landmark papers by Widom \cite{Wi75, Wi76} which used operator theory methods to extend the Szeg\H{o} theorem to the matrix case. The methods were then again used for singular symbols in \cite{Ba79} and \cite{Bo82}. After this, it seemed a whole industry of using operator methods to deal with more complicated singular symbols as well as other classes of operators was created. Much of the history can be found at the end of the chapters in \cite{BS06}. Also, while the present work deals with smooth symbols only, we note that Toeplitz determinants with matrix symbols possessing jumps were recently considered in  \cite{BEV}.

Here is the idea that first appeared in \cite{BH} that makes everything work. Suppose we have bounded operators $U$ and $L$ defined on a Hilbert space along with a sequence of  finite-dimensional orthogonal projections $P_{n}$ that converge strongly to the identity and also satisfy
\[ P_{n}L = P_{n}LP_{n} \,\,\,\,\mbox{and}\,\,\,\, UP_{n} = P_{n}UP_{n}.\] We assume both $L$ and $U$ are invertible. If we want to compute the determinant of an operator $P_{n}AP_{n},$ we notice
\[ P_{n}\,L\,L^{-1}\,A\,U^{-1} \,U\,P_{n} = P_{n}\,L\,P_{n}\,L^{-1}\,A\,U^{-1}\, P_{n}\,U\,P_{n} .\] Then if in addition we know that the operator $L^{-1}\,A\,U^{-1}$ is of the form $I + K$ where $K$ trace class, we have
\[ \det P_{n}\,A\,P_{n} \sim \det P_{n}\,L\,P_{n} \cdot \det P_{n}\,U\,P_{n} \cdot \det L^{-1}\,A\,U^{-1}.\] 
Of course one needs to be able to compute $\det P_{n}\,L\,P_{n} \det P_{n}\,U\,P_{n}.$  In all three of our above cases, this is easy to do and the asymptotic formula then just falls out. 

Sometimes even more can be said and an exact identity for determinants can be produced. Using the notation from above, we have
\[ \det P_{n} \,L^{-1}\,A\,U^{-1}\, P_{n} = \det L^{-1}\,A\,U^{-1} \cdot\det Q_{n} U A^{-1} L Q_{n} \] where
$Q_{n} = I - P_{n}.$ This follows directly form the Jacobi identity that says if an operator $B = I + K$ where $K$ is trace class, then for any projection $P = I - Q,$
\[  \det P \,B\, P = \det B \cdot \det Q B^{-1} Q. \]
Note in the above, the determinants are restricted to the images of the projections. 

Thus if one can find the inverse of $A$ easily and then write $U \,A^{-1} \,L $ in some useful way, one has not only an asymptotic formula for the determinants, but an exact expression. In the Toeplitz case, this leads to the identity known as the Borodin-Okounkov-Case-Geronimo Identity (BOCG). More will be said about this in the next section of the paper and later about generalizations of the identity.

Although the projections we will consider are finite dimensional they need not be. Thus this idea also works for Wiener-Hopf determinants and in other situations as well. A summary of the  Wiener-Hopf case results can be found in \cite{BS06}.

\section{Finite Toeplitz Matrices}
We begin with a $N \times N$ matrix-valued function $\phi$ defined on the unit circle $\T=\{z\in\C\,:\, |z|=1\}$ with Fourier coefficients
$$\phi_{k} = \frac{1}{2\pi} \int_{0}^{2\pi} \phi(e^{i\theta})\,e^{-ik\theta}\,\,d\theta ,$$
i.e., 
$$ \phi(e^{i\theta}) = \sum_{k=-\infty}^{\infty} \phi_{k} \,e^{ik\theta} =  \sum_{k=-\infty}^{\infty} \phi_{k} \,z^{k},$$
($z = e^{i\theta} $) 
and consider the matrix 
$$
T_n(\phi)\, = \, (\phi_{j-k})_{j,\,k\, = \,0,\, \dots,\, n-1}.
$$
The matrix $T_{n}(\phi)$ is referred to as the finite Toeplitz matrix with symbol $\phi.$ 
It has the form 
\[ \left[\begin{array}{ccccc}\phi_{0} & \phi_{-1} & \phi_{-2} &  \cdots&  \phi_{-(n-1)} \\
\phi_{1} & \phi_{0} & \phi_{-1} &  \cdots  & \phi_{-(n-2)} \\
\phi_{2} & \phi_{1} & \phi_{0} & \cdots  & \phi_{-(n-3)} \\
\cdots & \cdots & \cdots & \cdots  & \cdots \\
\cdots & \cdots & \cdots&\cdots & \cdots \\
\phi_{n-1} & \phi_{n-2} & \phi_{n-3} & \cdots & \phi_{0}
\end{array}\right].
\]
In the block case ($N\ge2$), each entry $\phi_{k}$ is an $N \times N$ matrix and $T_{n}(\phi)$ is of size $nN \times nN.$

Since the 1950s there has been considerable interest in the asymptotics of the determinants of Toeplitz matrices as $n \rightarrow \infty.$ The interest was motivated by the study of the two-dimensional Ising model correlations, but other important connections are related to many other areas of applied mathematics, such as entanglement computations in statistical physics or problems in random matrix theory, see for example \cite{GIKMV, IJK, IMM, JK} for some applications and \cite{BS99} and \cite{BS06} for more history and general facts about Toeplitz operators.

For finite Toeplitz matrices with well behaved symbols, the most well-known classical result is the Szeg\H{o}-Widom Limit Theorem which states that if the matrix valued symbol
$\phi$ defined on the unit 
circle $\T$ is such that $\det \phi$ has a sufficiently smooth logarithm then the determinant of the block Toeplitz
matrix $T_n(\phi)$  has the asymptotic behavior
$$  D_{n}(\phi) = \det T_n(\phi)  \, \sim G(\phi)^n\,E(\phi)\ \ \ {\rm as}\ n\to\iy.$$

Here are the constants:
$$G(\phi) = e^{(\log \det \phi)_0}$$
and
\begin{align*}
 E(\phi)  = \det \left( T(\phi)T(\phi^{-1}) \right)
\end{align*}
where
\begin{align}\label{Toe}
T(\phi) = (\phi_{j-k}),\,\,\,\,\,\,\,\,0 \leq j, k < \iy 
\end{align}
is the Toeplitz operator defined on $(H^{2})^{N} = H^{2} \oplus \cdots \oplus H^{2}$ where $H^{2}$ is the Hardy space of the circle. We note that the constant $E(\phi)$ was first discovered by Szeg\H{o} in the scalar case ($N =1$), although not in the form given above, and then extended by Widom to the matrix-valued case \cite{Wi74, Wi75, Wi76}.

To make sense of the term $\det \left( T(\phi)T(\phi^{-1}) \right)$ we should note that
we can always define the determinant of an operator of the form 
\[ I + K\]
where $K$ is a trace class operator defined on a Hilbert space. 
Such operators $K$ are compact with discrete eigenvalues $\lambda_{i}$ that satisfy 
\[ \sum_{i=0}^{\infty}|\lambda_{i}| < \infty \] and thus
\[ \det (I + K) = \prod_{i=0}^{\infty} (1+ \lambda_{i}) \]
is well defined.

To state the smoothness assumptions more precisely, 
let $\B$ stand for the set of all functions $\phi\in (L^1(\T))^{N\times N}$ such that the Fourier coefficients satisfy

 \[ \| \phi\|_{\B} :=\sum_{k =-\iy} ^{\iy} |\phi_{k}| + 
 \Big(\sum_{k = -\iy}^{\iy} |k|\cdot|\phi_{k}|^{2}\Big)^{1/2} < \iy. \]

With the norm (any matrix norm will do for the matrix Fourier coefficients),  and pointwise defined algebraic operations on $\T$, the set $\B$ becomes a Banach algebra of continuous functions on the unit circle.

The Szeg\H{o}-Widom Limit Theorem  holds provided that
$\phi\in \B$ and the function $\det \phi$ does  not vanish on $\T$ and has winding number zero.
It is known that this last requirement is equivalent to both $T(\phi)$ and $T(\phi^{-1})$ being Fredholm operators with 
index zero.
  
We will prove the Szeg\H{o}-Widom result by first proving an identity for the determinants, as we outlined in the introduction. This is the identity called the Borodin-Okounkov-Case-Geronimo (BOCG) identity. As we will see, the identity requires the invertibility of $T(\phi)$ and $T(\phi^{-1})$. Thus it does not yield quite as general a proof, but nonetheless it is an extremely insightful and useful identity.

To explain this identity, in addition to the Toeplitz operator, we also define a Hankel operator
\begin{align}
 H(\phi) &=  (\phi_{j+k+1}), \,\,\,\,\,\,\,\,0 \leq j,k < \infty.
 \label{Han}
\end{align}
For $\phi, \psi \in L^{\infty}(\mathbb{T})^{N \times N}$ the identities
\begin{align} \label{Tab}
T(\phi \psi) & = T(\phi)T(\psi) +H(\phi)H({\tilde \psi})\\[.5ex]
H(\phi \psi) & =T(\phi)H(\psi) +H(\phi)T({\tilde \psi})
 \end{align}
are well-known. Here, and in what follows,
$$
\tilde{\phi}( e^{i\,\theta}) = \phi (e^{-i\,\theta}).
$$
It follows from these identities that if  $\psi_{-}$ and $\psi_{+} $ have the property that all their Fourier coefficients vanish for $k > 0$ and $k < 0$, respectively, then 
\begin{align}\label{talpro}
T(\psi_{-} \phi \psi_{+} ) &= T(\psi_{-} )T(\phi) T( \psi_{+} ) ,
\\[.5ex]
H(\psi_{-} \phi \tilde{\psi}_{+} ) &= T(\psi_{-} )H(\phi) T( \psi_{+} ).
\end{align}

Suppose now we have a bounded and invertible matrix symbol $\phi$ such that 
\begin{align}\label{factors:uv}
\phi(z) & = u_{-}(z)u_{+}(z) = v_{+}(z)v_{-}(z),\qquad z\in\T,
\end{align}
and where $u_{+}, u_{+}^{-1}, v_{+}, v_{+}^{-1}$ and $u_{-}, u_{-}^{-1}, v_{-}, v_{-}^{-1}$ have  
the property that all their Fourier coefficients vanish for $k > 0$ and $k < 0$, respectively. The above factorizations are called the (left/right) canonical Wiener-Hopf factorizations of $\phi$. 
In the block case ($N\ge2$) the left and right factorizations may be different. For general symbols
it is known that $T(\phi)$ (resp.~$T(\phi^{-1})$) is invertible if and only if $\phi$ has a left (resp.,\ right) factorization with some additional (rather complicated) conditions on the factors (see \cite{CG, LS} for more information and general factorization theory).

In the case of invertible matrix symbols $\phi\in\B$, the situation is slightly simpler.
The left/right canonical factorization of $\phi$ exists if and only if $T(\phi)$, resp.~$T(\phi^{-1})$, is invertible
and in this case the factors and their inverses belong to $\B$ as well (see \cite[Sect.~10.3]{BS06}).
It is a cruciual observation that the inverses of the Toeplitz operators can then be written as
\begin{align}\label{Tinv}
T(\phi)^{-1}=T(u_+^{-1})T(u_{-}^{-1})\quad\mbox{ and }\quad 
T(\phi^{-1})^{-1}=T(v_+)T(v_{-}).
\end{align}
Indeed, the algebra property (\ref{talpro}) and  the factorization of $\phi$ shows that 
\[T(\phi) = 
T(u_-u_+) = T(u_{-}) \,T(u_{+}) ,\]
\[
T(\phi^{-1})=T(v_-^{-1}v_+^{-1})=T(v_-^{-1})T(v_+^{-1})
\]
and 
\[ T(v_{+}^{-1})T( v_{+}) =   T(v_{+})T( v_{+}^{-1}) = I\]
\[ T(v_{-}^{-1})T( v_{-})  =   T(v_{-})T( v_{-}^{-1}) = I\]
and a corresponding identity for $u_+$ and $u_-$. This verifies \eqref{Tinv}.

Now we are prepared to prove the BOCG identity.

\begin{theorem}
Suppose $\phi\in\B$ is invertible and in addition, the operators $T(\phi)$ and $T(\phi^{-1})$ are invertible.  Then the BOCG identity reads
\begin{eqnarray*}\label{BOCG}
\det T_n(\phi) &=& G(\phi)^n \,\,E(\phi) \cdot 
\det\left(I-H(z^{-n}v_{-}u_{+}^{-1})H(\tilde{u}_{-}^{-1} \tilde{v}_{+}z^{-n})\right),
\end{eqnarray*}
where the factors are given by \eqref{factors:uv}.
\end{theorem}

\proof
We let 
$P_{n}$ be the projection defined on $(H^{2})^{N}$ that maps 
$$\{x_k\}_{k=0}^\infty \,\,\,\mbox{to}\,\,\,\{x_0,\dots,x_{n-1},0,0, \dots\}
$$ identifying a function in $(H^{2})^{N}$ with its Fourier coefficients (in $\C^N$), and we define $Q_{n} = I - P_{n}.$
Identifying $z^{\pm n}=z^{\pm n }I_N$, it is easy to check that $$Q_{n} = T(z^{n})T(z^{-n}), \,\,\,I =  T(z^{-n})T(z^{n}).$$

In what follows and for the rest of the paper, determinants are always defined on the appropriate space. Thus for example, 
$\det P_{n} A P_{n}$ is the determinant defined on the image of $P_{n}.$ 

Next we make some simple observations:

1) If $\phi_{k} = 0$ for all $k < 0$ or for all $k > 0, $ then the Toeplitz matrices are triangular and $D_{n}(\phi) = \det T_{n}(\phi) = (\det \phi_{0})^{n}.$

2) $T_{n}(\phi) = P_{n}T(\phi)P_{n}$, i.e., $T_{n}(\phi)$ is the upper left corner of the Toeplitz operator $T(\phi).$

The projection $P_{n}$ has the nice property mentioned in our introduction that if
$U$ is an operator whose matrix representation has an upper triangular form, then
\[P_{n} U P_{n} = U P_{n}.\] 
And if $L$ is an operator whose matrix representation has a lower triangular form, then
\[ P_{n} L P_{n} =  P_{n} L. \]
  
We let $\phi = u_{-}u_{+} = v_{+}v_{-}.$   
Observing that $T(v_+)$ is of lower and $T(v_-)$ is of upper triangular form, we now write
\begin{align*}
P_{n} T(\phi) P_{n} 
&= P_{n} T(v_{+}) T(v_{+}^{-1})T(\phi) T(v_{-}^{-1}) T(v_{-})P_{n} \\
  &= P_{n} T(v_{+}) P_{n}T(v_{+}^{-1})T(\phi) T(v_{-}^{-1}) P_{n}T(v_{-})P_{n}.
\end{align*}
Taking determinants, we have $$\det P_{n} T(v_{+}) P_{n} = (\det ((v_{+})_{0}))^{n}$$ and $$\det P_{n}T(v_{-}) P_{n} = (\det ((v_{-})_{0}))^{n}.$$
 So at this point we have
 \[ D_{n}(\phi) = G(\phi)^{n} \det  P_{n}T(v_{+}^{-1})
 T(\phi)T(v_{-}^{-1}) P_{n} .\]
   
The term 
\[ P_{n} \,T(v_{+}^{-1}) T(\phi)  
T(v_{-}^{-1})\, P_{n} \] 
is of the form 
\[ P_{n} (I + K) P_{n}\]
where $K$ is trace class. To see this we note that an operator is Hilbert-Schmidt if, in matrix form, its entries satisfy
\[ \sum_{i,j \geq 0} |a_{i,j}|^{2} < \infty ,\] and that it is well-known that the product of two Hilbert-Schmidt operators is trace class. A good reference for this fact and general properties of trace class and Hilbert-Schmidt operators is \cite{GK}. Hankel operators with symbols taken from $\B$ are Hilbert-Schmidt because of the norm condition.

Indeed, we can write the operator $A=T(v_{+}^{-1})T(\phi) T(v_{-}^{-1})$ as
\begin{align*}
A  &= T(v_{+}^{-1})T(\phi) T(v_{-}^{-1})  T(v_+^{-1})  T(v_+)\\
 &= T(v_{+}^{-1})T(\phi)T(v_{-}^{-1} v_+^{-1}) T(v_+) \\
 &= T(v_{+}^{-1})T(\phi)T(\phi^{-1}) T(v_+)\\
 &= T(v_+^{-1}) \Big(T(I_N)-H(\phi)H(\tilde{\phi}^{-1})\Big)T(v_+)\\
 &= I-   T(v_+^{-1}) H(\phi)H(\tilde{\phi}^{-1})T(v_+)
 \end{align*}
 using \eqref{Tab}
with the product of the Hankel operators being trace class. This implies that 
the operator determinant 
$$
\det A= \det T(\phi)T(\phi^{-1})=\det\big(I-H(\phi)H(\tilde{\phi}^{-1}))=E(\phi)
$$
is well-defined.

The operator 
$$
A=T(v_{+}^{-1})T(\phi) T(v_{-}^{-1})=T(v_+)T(u_-)T(u_+)T(v_{-}^{-1})
$$
is invertible and the inverse is
\begin{align*}
A^{-1} &= T(v_-) T(u_+^{-1})T(u_-^{-1})  T(v_+)
\\
&= T(v_-u_+^{-1})T(u_{-}^{-1}v_+),
\end{align*}
and it is also of the form identity plus trace class operator,
\begin{align*}
A^{-1} &=I-H(v_-u_+^{-1}) H(\tilde{u}_-^{-1}\tilde{v}_+),
\end{align*}
since $v_-u_+^{-1}u_{-}^{-1}v_+=I_N$ and the Hankel operators are Hilbert-Schmidt.

To obtain the next step we use the Jacobi identity 
\[ \det P_{n}A P_{n} = \det A\cdot\det (Q_{n}A^{-1} Q_{n})\]
in which $Q_{n}=I-P_{n}$ and  which is true whenever $A$ is an invertible operator of the form identity plus trace class.

This holds because
\begin{align*}
\det P_{n}A P_{n} &= \det (Q_{n}+P_{n}A P_{n})=\det (Q_{n}+A P_{n})\\
&= \det A\cdot \det (A^{-1} Q_{n}+P_{n})=
\det A\cdot\det (Q_{n}A^{-1} Q_{n}+P_n).\\
&=
\det A\cdot\det (Q_{n}A^{-1} Q_{n}).
\end{align*}

We apply Jacobi's identity to our $A$ and it  follows that
\begin{align*}
\det P_{n}T(v_{+}^{-1})T(\phi) T(v_{-}^{-1})P_{n}
&= E(\phi)\cdot 
\det Q_n A^{-1} Q_n
\end{align*}
where
\begin{flalign*}
 \det Q_{n}A^{-1}Q_{n} 
 & =  \det( P_{n} + Q_{n}(I - H(v_{-}u_{+}^{-1} )H(\tilde{u}_{-}^{-1} \tilde{v}_{+}))Q_{n} )&\\
 & = \det( I - Q_{n}H(v_{-}u_{+}^{-1} )H(\tilde{u}_{-}^{-1} \tilde{v}_{+})Q_{n} )&\\
 & = \det( I - T(z^{n})T(z^{-n})H(v_{-}u_{+}^{-1} )H(\tilde{u}_{-}^{-1} \tilde{v}_{+})T(z^{n})T(z^{-n}) )&\\
 & = \det( I - T(z^{-n}) H(v_{-}u_{+}^{-1} )H(\tilde{u}_{-}^{-1} \tilde{v}_{+})T(z^{n} ))&\\
 & = \det( I - H(z^{-n} v_{-}u_{+}^{-1} )H(\tilde{u}_{-}^{-1} \tilde{v}_{+})z^{-n}) )
\end{flalign*}
and the identity is proved.
\endproof

For other proofs (including two original ones) we refer the reader to \cite{BW, BO, Bo01, GC}.

From this BOCG identity we have an instant proof of the Szeg\"{o}-Widom theorem. Notice that $Q_{n}f$ tends to zero for any fixed $f$ as $n \rightarrow \infty$ and this allows one to say that 
$$Q_{n}\,H(v_{-}u_{+}^{-1})H(\tilde{u}_{-}^{-1} \tilde{v}_{+})\,Q_{n}$$ tends to zero in the trace norm. Thus the determinant
$$ \det ( I  - Q_{n}\,H(v_{-}u_{+}^{-1})H(\tilde{u}_{-}^{-1} \tilde{v}_{+})\,Q_{n}) $$ tends to one and 
$$  D_{n}(\phi) = \det T_n(\phi)  \, \sim G(\phi)^n\,E(\phi)\ \ \ {\rm as}\ n\to\iy.$$
For more details see \cite{BW}.

In the scalar case, the constant $E(\phi)$ has a nice concrete description. We will not give a proof of this here, but refer to \cite{Wi75}.  There it was first shown that 
$$\det T(\phi)T(\phi^{-1}) =  \exp\left(\sum_{k=1}^{\iy}k\,(\log\phi)_k\,(\log\phi)_{-k}\right).$$ We should point out that the formula on the right was known much before the 1975 paper and proved first by Szeg\H{o}. 

In the block case no such description is in general available and the constant may even vanish. However, sometimes more can be said. This follows again from the BOCG identity.

\begin{theorem}
Let $\phi \in \B$ and suppose $T(\phi)$ and $T(\phi^{-1})$ are invertible.  Assume that $\phi_k=0$ for all $k> n$ or that $\phi_{-k}=0$ for all $k> n$. Then
$$E(\phi)=G(\phi)^n \det T_n(\phi\iv).$$
\end{theorem}
\begin{proof}
Since $E(\phi) = E(\phi^{-1})$ we apply BOCG to $\phi^{-1}$.
Notice also the BOCG identity can be restated as follows by using the algebra properties \eqref{talpro},
\begin{eqnarray*}
\det T_n(\phi\iv) &=& \frac{E(\phi)}{G(\phi)^n}\cdot 
\det\left(I-H(z^{-n}\phi)T\iv(\tilde{\phi})H(\tilde{\phi} z^{-n})T\iv(\phi)\right).
\end{eqnarray*}
 Suppose $\phi_k=0$ for all $k> n.$ Then the Fourier coefficient of $z^{-n}\phi$ vanishes for $k>n$ and the Hankel operator
 $H(z^{-n}\phi)$ is the zero operator. If $\phi_{-k}=0$ for  $k> n,$ then $H(\tilde{\phi} z^{-n})$ is the zero operator for $k> n$ and the result follows.
\end{proof}

As we mentioned earlier, in the 1976 Widom paper the Szeg\H{o}-Widom result was proved more generally than we have proved it here. Our assumptions on the factorization of our symbol needed for the BOCG identity imply that our Toeplitz operator is invertible. Widom proved the result with only the assumption that the operator was Fredholm with index zero.
 
Now we turn to a different class of structured operators.

\section{Finite Toeplitz-plus-Hankel operators}

We first consider the general problem of finite matrices of the form
\[ P_{n} \big( T(\phi) + H(\psi) \big) P_{n},\] where both $\phi$ and $\psi$ are matrix-valued symbols. 
The matrix representations of the operators are of the form
\[ (\phi_{j-k} + \psi_{j+k +1}),\qquad 0\le j,k<\infty.\] 

Determinants of such operators in some special cases arise naturally in applications. For example, if $\psi = -\phi$ and if $\phi = g\,\tilde{g}$ for appropriately defined $g,$ the determinant of $P_{n}( T(\phi) + H(\psi) )P_{n}$ computes the average of $\det g(U)$ over the subgroup of orthogonal matrices of size $2n$ with determinant $1.$ Other averages over the classical compact groups can be computed using different choices of $\psi$. These averages are important in random matrix theory and number theory \cite{BR, CFS, CS07}.

 We begin with a more general theorem.
\begin{theorem}\label{g.th}
Suppose that $\phi\in\B$ is invertible, the operator $T(\phi^{-1})$ is invertible and 
the operator $K$ is trace class. Then
\[ \det P_{n} (T(\phi) + K)P_{n} \sim G(\phi)^{n} \det \big( T(\phi^{-1})(T(\phi) + K)\big) .\]
\end{theorem}

\begin{proof}
This proof uses the same idea as the proof for BOCG. Because $T(\phi^{-1})$ is invertible we have a Wiener-Hopf factorization $\phi = v_{+}v_{-}$. Now
\[ P_{n} ( T(\phi) + K) P_{n} = P_{n} T(v_{+}) T(v_{+}^{-1}) (T(\phi) + K )T(v_{-}^{-1}) T(v_{-}) P_{n}\]
\[ = P_{n} T(v_{+}) P_{n}T(v_{+}^{-1}) (T(\phi) + K )T(v_{-}^{-1}) P_{n}T(v_{-}) P_{n}, \] and thus we have
\[ \det P_{n} (T(\phi) +K )P_{n} = G(\phi)^{n} \det P_{n }T(v_{+}^{-1}) (T(\phi) + K )T(v_{-}^{-1}) P_{n}.\]
By our assumptions $T(v_{+}^{-1}) (T(\phi) + K )T(v_{-}^{-1})$ is of the form $I + L$ where $L$ is trace class and thus 
\[ \det P_{n }T(v_{+}^{-1}) (T(\phi) +K) T(v_{-}^{-1} )P_{n} \] converges to
\[ \det T(v_{+}^{-1}) (T(\phi) +K )T(v_{-}^{-1}). \]
This evaluates to 
\[ \det T(v_{-}^{-1})T(v_{+}^{-1}) (T(\phi) + K ) = \det T(\phi^{-1})(T(\phi) + K )\]
and concludes the proof.
\end{proof}

In the special case $K=0$, the  previous result comes down to the Szeg\H{o}-Widom theorem.
Taking for $K$ a trace class Hankel operator, the  Toeplitz-plus-Hankel asymptotics is obtained, 
which differs  from the Toeplitz case only in a constant term as can be seen from the following.

\begin{corollary}
Suppose that $\phi\in\B$ is invertible, the operator $T(\phi^{-1})$ is invertible and $H(\psi)$ is trace class. Then
\[ \det P_{n} (T(\phi) + H(\psi))P_{n} \sim G(\phi)^{n} \det \big( T(\phi^{-1})(T(\phi) + H(\psi))\big) .\]
If, in addition, $T(\phi)$ is invertible, then
\[ \det (P_{n} (T(\phi) + H(\psi)) P_{n}) \sim G(\phi)^{n} E(\phi) \det  (I +T(\phi)^{-1} H(\psi) ) .\]
\end{corollary}

While the above result yields a general asymptotic formula, it is not in some sense complete. One would like a BOCG type formula and in the scalar case it would be useful to have more concrete descriptions for the constants.

We begin with identities similar to the BOCG type proved in the previous section for an abstract class of operators generated by matrix-valued symbols. The proofs of these are similar (but not exactly the same) as the one done in the previous section and the proofs and additional details can be found in \cite{BE09} where the scalar cases are done.

Let $\cS$ stand for a unital Banach algebra of functions on the unit circle 
continuously embedded into $L^\iy(\T)^{N \times N}$ with norm $\|\cdot \|_{\cS}$ and with the property that 
$\phi \in \cS$ implies that $\tilde{\phi}\in\cS$ and $P \phi\in \cS$. Moreover, define
\begin{align*}
\cS_- &= \Big\{ \phi\in \cS\;:\; \phi_n=0 \mbox{ for  all } n>0\Big\},\nn\\
\cS_{+}&=\Big\{ \phi\in \cS\;:\; \phi_n=0 \mbox{ for  all } n<0\Big\},\nn\\
\cS_0 &= \Big\{ \phi\in \cS\;:\; \phi=\tilde{\phi} \Big\}.\nn
\end{align*}

In the above $\tilde \phi (e^{i\theta}) = \phi(e^{-i\theta})$
and $P$ is the Riesz projection. 
 
Assume that $M: \phi \in L^\iy(\T)^{N\times N} \mapsto M(\phi)\in \cL((H^{2})^{N})$ (the set of bounded operators on $(H^{2})^{N}$) is a continuous linear map such that:
\begin{enumerate}
\item[(a)]
If $\phi \in \cS$, then $M(\phi)-T(\phi)$ is trace class
and there exists some constant $C$ such that
$$ 
\|M(\phi)-T(\phi)\|_{1} \le C\, \|\phi\|_{\cS}.
$$
\item[(b)]
If $\psi \in \cS_-$, $\phi \in \cS$, $ \gamma \in \cS_0$, then
$$
M(\psi \phi \gamma)=T(\psi)M(\phi) M(\gamma).
$$
\item[(c)]
$M(1)=I$.
\end{enumerate}
Then we say $M$ and $\cS$ are {\em compatible pairs}. Note in the above $\| \cdot \|_{1}$ is the trace norm.  
 
All of the following, which occur sometimes in applications, can be realized as compatible pairs:

\begin{enumerate}
\item[(I)]
$M(\phi) =T(\phi)+H(\phi)$,
\item[(II)]
$M(\phi) =T(\phi)-H(\phi)$,
\item[(III)]
$M(\phi)=T(\phi)-H(z^{-1} \phi)$, 
\item[(IV)]
$M(\phi)=T(\phi)+H( z \,\phi)Q_{1}$.
\end{enumerate}

A Banach algebra that satisfies the compatible pair conditions for all four of the above cases is the set of all $\phi$ such that 
\[ |\phi_{0}| + \|H(\phi)\|_{1}+ \|H(\tilde{\phi})\|_{1} < \infty.\] 
It is known that the trace class condition on the Hankel operators $H(\phi)$ and $H(\tilde{\psi})$ is equivalent to requiring 
that $\phi$ belongs to the Besov class $B_1^1$ (see \cite[Thm.~10.10]{BS06} and the reference therein). 
A more convenient sufficient conditions is to require that the symbol satisfies 
$$
\|\phi\|_{F\ell_2^1}=\left(\sum_{k=-\infty}^\infty (1+|k|)^2|\phi_k|^2\right)<\infty.
$$

In most applications the symbol $\phi$ is even and we restrict ourselves now to that case in what follows. 
 
\begin{theorem}\label{BEf}
Let $M$ and $\cS$ be a compatible pair, with $\phi_{\pm1}$ and $\phi_{\pm}^{-1}$ in $\cS_{\pm}.$ Let $\phi = \phi_{+}\phi_{-}$ and assume that $\phi$ is even.   Then
\[ \det P_{n} M(\phi) P_{n}  = G(\phi)^{n} E_M(\phi) \det (I + Q_{n} \,K \,Q_{n} ),\]
where
\[ E_M(\phi) =  \det T(\phi^{-1}) M(\phi)  \]
and $K = M(\phi_{+}\iv) T(\phi_{+}) - I$ and $K$ is trace class.
\end{theorem}
 \proof
 We sketch the proof since it follows the one in the previous section. We have
 \[ \det P_{n} M(\phi) P_{n} = G(\phi)^{n} \det P_{n} T(\phi_{+}^{-1}) M(\phi) T(\phi_{-}^{-1}) P_{n} \]
 \[ = G(\phi)^{n} \det T(\phi_{+}^{-1}) M(\phi) T(\phi_{-}^{-1}) \det Q_{n} T(\phi_{-}) M(\phi)^{-1} T(\phi_{+}) Q_{n}.\]
 The determinant $\det T(\phi_{+}^{-1}) M(\phi) T(\phi_{-}^{-1}) $ is defined. This is because \[ T(\phi_{+}^{-1}) M(\phi) T(\phi_{-}^{-1}) = T(\phi_{+}^{-1}) (M(\phi) - T(\phi) + T(\phi)) T(\phi_{-}^{-1}) \]
 and $M(\phi) - T(\phi)$ is assumed to be trace class, and in addition, from Proposition 2.1 of \cite{BE09}, we know that $T(\phi_{+}^{-1})T(\phi) T(\phi_{-}^{-1})$ is of the form $I$ plus trace class. Now since our symbol is even $M(\phi)^{-1} = M(\phi^{-1})$ and thus
 \[ \det Q_{n} T(\phi_{-}) M(\phi)^{-1} T(\phi_{+}) Q_{n} = \det Q_{n} T(\phi_{-}) M(\phi^{-1}) T(\phi_{+}) Q_{n}\]
 \[ = \det Q_{n} T(\phi_{-}) M(\phi_{-}^{-1}\phi_{+}^{-1}) T(\phi_{+}) Q_{n} = \det Q_{n} M(\phi_{+}^{-1}) T(\phi_{+}) Q_{n}\]
 \[ = \det Q_{n} ((M(\phi_{+}^{-1}) - T(\phi_{+}^{-1}) + T(\phi_{+}^{-1}) T(\phi_{+}) Q_{n}\] and the result follows.
 \endproof
 
 For our four cases of interest, our operator $K$ can be expressed in terms of Hankel operators.
 \begin{enumerate}
\item[(I)]
$K = H(\phi_{+}^{-1})T(\phi_{+}) = H(\phi_{+}^{-1}\tilde \phi_{+})$
\item[(II)]
$K =   -H(\phi_{+}^{-1})T(\phi_{+}) = -H(\phi_{+}^{-1}\tilde \phi_{+}).$
\item[(III)]
$K = - H(z^{-1}\phi_{+}^{-1})T(\phi_{+}) =   - H(z^{-1}\phi_{+}^{-1} \tilde{\phi_{+}})$
\item[(IV)]
$K = H(z \phi_+\iv  \tilde{\phi}_+)-T(\phi_+\iv)H(z\tilde{\phi}_+)$
\end{enumerate}
 
While just as in the finite Toeplitz setting, the constant in the matrix case may be hard to explicitly describe, for scalar functions the answer was obtained in \cite{BE09}. We only state the result here again for $\phi$ even.

Suppose $\phi = \exp \beta.$ Then
\[ E(\phi) =  \exp\Big( \tr(M( \beta)-T(\beta)) + \frac{1}{2}\tr\, H(\beta)^2\Big).\] In our special cases, we have that $E(\phi)$ is given by

\begin{enumerate}
\item[(I)]
$\exp\Big( \sum_{k =1}^{\infty} \beta_{2k+1}  +\frac{1}{2}\sum_{k=1}^\iy k \,\beta_{k}^{2}\Big)$
\item[(II)]
$\exp\Big(- \sum_{k =1}^{\infty} \beta_{2k+1}  +\frac{1}{2}\sum_{k=1}^\iy k \,\beta_{k}^{2}\Big)$
\item[(III)]
$\exp\Big(- \sum_{k =1}^{\infty} \beta_{2k}  +\frac{1}{2}\sum_{k=1}^\iy k \, \beta_{k}^{2}\Big)$
\item[(IV)]
$\exp\Big(\sum_{k=1}^\iy \beta_{2k}+\frac{1}{2}\sum_{k=1}^\iy k \, \beta_k^2\Big)$
\end{enumerate}

While the non-even cases may not occur as frequently in applications,  for completeness sake we include the case of non-even matrix valued $\phi$. 

\begin{theorem}
Let $[M,\cS]$ be a compatible pair, and let 
$\phi =v_{+}v_{-}$ with $v_{\pm}, v_{\pm}^{-1} \in \cS_{\pm}$. Moreover, assume that 
\begin{align}\label{anti-fact}
\phi\tilde{\phi}^{-1}=\xi_-\tilde{\xi}_-^{-1}
\end{align}
with $\xi_-,\xi_-^{-1}\in\cS_-$. Then 
\[ \det P_{n} M(\phi) P_{n}  = G(\phi)^{n} E_M(\phi) \det (I + Q_{n} K Q_{n} ),\]
with
\[ E_M(\phi) =   \det  T(\phi^{-1}) M(\phi),\] 
and $K = M(v_+^{-1}\xi_-) T(\xi_-^{-1} v_+) - I$ being trace class.
\end{theorem}
\proof
The first steps of this proof are the same as before. We have
\[ \det P_{n} M(\phi) P_{n} = G(\phi)^{n} \det P_{n} T(v_{+}^{-1}) M(\phi) T(v_{-}^{-1}) P_{n} \] 
and consider $A=T(v_+^{-1})M(\phi)T(v_-^{-1})$. We can write it as
$$
A= T(v_-)T(\phi^{-1})M(\phi)T(v_-^{-1})
$$
concluding that it is of the form identity plus trace class. In particular, $\det A=E_M(\phi)$.
To invert $A$ we need to find the inverse of $M(\phi)$. Define $\phi_0=\xi_-^{-1}\phi$ and now \eqref{anti-fact} implies that 
$$
\tilde{\phi}_0=\tilde{\xi}_-^{-1} \tilde{\phi}=\xi_-^{-1}\phi=\phi_0.
$$
Therefore, $\phi=\xi_-\phi_0$ with $\phi_0,\phi_0^{-1}\in\cS_0$ entailing
$$
M(\phi)=T(\xi)M(\phi_0),\qquad M(\phi)^{-1} =M(\phi_0^{-1})T(\xi_-^{-1}).
$$
Here we used condition (b) in the definition of a compatible pair.
The inverse of $A$ exists and is given by
\begin{align*}
A^{-1} &=T(v_-)M(\phi)^{-1}T(v_+)=T(v_-)M(\phi_0^{-1})T(\xi_-^{-1})T(v_+)\\
&=M(v_-\phi_0^{-1})T(\xi_-^{-1}v_+)
=M(v_+^{-1}\xi_-)T(\xi_-^{-1}v_+)
\end{align*}
noting that $v_-\phi_0^{-1}=v_+\phi\phi_0^{-1}=v_+^{-1}\xi$. Jacobi's identity
and
$$
\det Q_n A^{-1} Q_n =\det( P_n +Q_nA^{-1}Q_n)=\det(I+ Q_n K Q_n)
$$
with $K$ as above concludes the proof.
\endproof

The factorization \eqref{anti-fact} of the matrix function $\phi\tilde{\phi}^{-1}$ is also a 
canonical Wiener-Hopf factorization because $\tilde{\xi}_-,\tilde{\xi}_-^{-1}\in\cS_+$.
Indeed, under the additional assumption that $\cS$ contains only continuous functions, we can conclude that each such Wiener-Hopf factorization of $\phi\tilde{\phi}^{-1}$ in $\cS$ is of this form: Assuming $\phi\tilde{\phi}^{-1}=\xi_-\xi_+$ implies
$$
\tilde{\xi}_+^{-1}\tilde{\xi}_{-}^{-1}= (\widetilde{\phi\tilde{\phi}^{-1}})^{-1}=\phi\tilde{\phi}^{-1}=\xi_-\xi_+.
$$
Hence $G=\xi_+\tilde{\xi}_-=\xi_-^{-1}\tilde{\xi}_+^{-1}$ with some invertible matrix $G$.
But then $\phi\tilde{\phi}^{-1}=\xi_- G \tilde{\xi}_-^{-1}$. 
This reads, 
$$\phi(z)\phi(z^{-1})^{-1}=\xi_-(z) G \xi_-(z^{-1})^{-1},\qquad z\in\T.$$ 
Continuity at $z=1$ or $z=-1$ implies that 
$G=I_N$, thus $\phi\tilde{\phi}^{-1}=\xi_- \tilde{\xi}_-^{-1}$ which is \eqref{anti-fact}.
We do not know whether the proceeding observation is true for every compatible pair, i.e.,
without the assumption that $\cS$ is contained in the set of continuous matrix functions.

As the proof indicates, the existence of the factorization \eqref{anti-fact} implies
the invertibility of $M(\phi)$. For certain compatible pairs,  the converse may also be true. 
In this connection we remark that the relationship between
factorizations and invertibility of certain Toeplitz-plus-Hankel operators 
$T(\phi)+ H(\phi W)$ (with $W$ a matrix satisfying $W^2=I_N$) has been examined in \cite{E04a} (see also  \cite[Chapt.~4]{E04b}).

\section{Finite sections of functions of Toeplitz operators}

Our goal in this section is to find the asymptotic formula for $\det P_{n}\, f(T(\phi) )\,P_{n}$ where $f$ is some appropriately defined function of the operator $T(\phi).$ This was motivated by a question to the authors from Maurice Duits in relation to some linear statistics problems which occur in random matrix theory. In that case the function $f$ was the standard exponential function. 

 \begin{proposition}
Suppose that $\phi \in \B$ and in addition $f$ is complex-valued and analytic in an open set in the plane containing the spectrum of  $T(\phi )$.
Then  
\[ f(T(\phi) ) - T(f(\phi) )\]
is a trace class operator. 
\end{proposition}
\begin{proof}
It is well-known that if $T(\phi)$ is invertible then so is $\phi$ and
\[ T(\phi)^{-1} - T(\phi^{-1}) \]
is trace class. This follows from the fact that 
\[ T(\phi)T(\phi^{-1}) = I - H(\phi) H(\widetilde{\phi}^{-1}) \]
is the identity plus a trace class operator. 
Multiplying from the left
by $T(\phi)^{-1}$ then shows that the above difference is trace class. 
 
To arrive at the result for a general $f$ we write, using functional calculus, 
\begin{align*}
 f(T(\phi) ) - T(f(\phi) ) 
 &= \frac{1}{2\pi i} \int_{\cC} f(z)\, ( z - T(\phi))^{-1}\, dz - \frac{1}{2\pi i} T(\int_{\cC} f(z)\, ( z - \phi)^{-1}\, dz)
 \\[1ex]
&= 
\frac{1}{2\pi i} \int_{\cC} f(z) \Big( (T(z-\phi))^{-1} - T((z-\phi)^{-1})\Big)\, dz
\end{align*}
and then use a standard approximation argument in connection with the preceding observation for $z-\phi$ instead of $\phi$. Here $\cC$ is any contour containing the spectrum of $T(\phi)$ which is known to also contain the spectrum of $\phi$.
\end{proof}

Now from this simple proposition, in connection with Theorem \ref{g.th}, we can prove our general result. 

\begin{theorem}
Suppose $\phi \in \B$ and $f$ is analytic complex-valued and  analytic on an open neighborhood of the spectrum of $T(\phi)$.
Moreover, assume that the function $f(\phi)\in\B$ and the operator $T(f(\phi)^{-1})$ are invertible. Then
\[  \det P_{n} f(T(\phi) )P_{n}\, \sim G(f(\phi))^n\,\det f(T(\phi))T((f(\phi)^{-1})\ \ \ {\rm as}\ n\to\iy,\]
where the operator determinant is well-defined.
\end{theorem}
\proof
By what we have shown above,
$$
K=f(T(\phi))-T(f(\phi))
$$
is trace class. Hence we can write 
$$
\det P_n f(T(\phi)) P_n = \det P_n(T(f(\phi))+K)P_n
$$
and see that we are in the position to apply Theorem \ref{g.th} (with $f(\phi)\in\B$ instead of $\phi\in\B$).
This provides the asymptotics
$$
\det P_n (T(\phi)) P_n \sim G(f(\phi))^n \det \big(T(f(\phi)^{-1})(T(f(\phi))+K)\big), 
\quad n\to \infty,
$$
where the operator determinant takes the expression given above. 

That this operator determinant is well-defined can also be seen directly by remarking that both $K$  and 
$$L= T(f(\phi) )T(f(\phi)^{-1})-I=-H(f(\phi))H(\widetilde{f(\phi)}^{-1}) $$
are trace class operators. Therefore,
$$
f(T(\phi))T(f(\phi)\iv)=(T(f(\phi))+K)T(f(\phi\iv)=I+L+K T(f(\phi\iv))
$$
is of the form identity plus a trace class operator. 
\endproof

As an application, if $f(z) = e^{z},$ then 
 \[ \det P_{n} e^{T(\phi)}P_{n}\, \sim G(e^{\phi})^n\,\det e^{T(\phi)}T(e^{-\phi})\ \ \ {\rm as}\ n\to\iy.\] 
 
 We should point out that the constant $\det e^{ T(\phi)}T(e^{-\phi})$ is important in itself. If one can compute 
 $$\det e^{T(\phi)}T(e^{-\phi})\,\, \,\mbox{and} \,\,\,\det e^{-T(\phi)}T(e^{\phi}),$$ then one can compute 
 $$\det e^{T(\phi)}T(e^{-\phi})\,\det e^{-T(\phi)}T(e^{\phi}) = E(\psi)$$ where $\psi = e^{ \phi}.$ This idea was used in other constant computations approaches, for example see \cite{BE07}. In particular, in the above scalar case, the constant can be computed
\[   \det e^{T(\phi)}T(e^{-\phi}) = \exp \Big( \frac{1}{2} \sum_{k = 1} ^{\infty} k \phi_{k} \phi_{-k} \Big). \]
For details we refer to \cite{E03a,E03b}. Therefore, in the scalar case, the asymptotic formula
$$
\det P_n e^{T(\phi)} P_n \sim G(e^{\phi})^n\exp( \Big( \frac{1}{2} \sum_{k = 1} ^{\infty} k \phi_{k} \phi_{-k} \Big), \qquad n\to \infty
$$
is obtained for $\phi\in \B$.


\begin{thebibliography}{XXXXX} 
 \bibitem[BR]{BR} 
 {\sc J. Baik,  E.M. Rains}, 
  Algebraic aspects of increasing subsequences,
  {\em Duke Math. J. } {\bf 109 } (2001),  no. 1, 1--65.


\bibitem[Ba79]{Ba79} 
{\sc E.L. Basor,}
A localization theorem for Toeplitz determinants,
{\em Indiana Univ. Math. J.} {\bf  28} (1979), no. 6, 975--983. 


\bibitem[BE07]{BE07}
{\sc E.L. Basor and T. Ehrhardt,}
Asymptotics of Block Toeplitz Determinants and the Classical Dimer Model,
{\em Commun. Math. Phys.} {\bf 274} (2007), 427--455.


\bibitem[BE09]{BE09}  
{\sc E.L. Basor, T. Ehrhardt, }
Determinant computations for some classes of Toeplitz-Hankel matrices, 
{\em Oper.~Matrices} {\bf 3}, no.~2 (2009), 167--186.

\bibitem[BEV]{BEV}  
{\sc E.L. Basor, T. Ehrhardt, J.A. Virtanen}, Asymptotics of block Toeplitz determinants with piecewise continuous symbols, 
{\em Comm. Pure Appl. Math.} (in press) arXiv:2307.00825.


\bibitem[BH]{BH}
{\sc E. Basor, J. William Helton,}
A new proof of the Szeg\"o limit theorem and new results for Toeplitz operators with discontinuous symbol,
{\em J. Operator Theory} {\bf 3}, no.~1(1980), 23--39.

 \bibitem[BW]{BW}
{\sc E.L. Basor, H. Widom,}  
On a Toeplitz determinant identity of Borodin and Okounkov,
{\em  Integral Equations Operator Theory} {\bf 37}, no.~4 (2000), 397--401.




\bibitem[BO]{BO}
{\sc A. Borodin, A. Okounkov,}
 A Fredholm determinant formula for Toeplitz determinants,
 {\em  Integral Equations Operator Theory} {\bf 37} (2000), 386--396.


\bibitem[Bo82]{Bo82}
{\sc A. B\"ottcher}, 
Toeplitz determinants with piecewise continuous generating functions,
{\em Z. Anal. Anwendungen} {\bf 1}, no.~2, 23-39 (1982).


\bibitem[Bo01]{Bo01}
{\sc A. B\"ottcher},  One more proof of the Borodin-Okounkov formula for Toeplitz determinants,
 {\em  Integral Equations Operator Theory} {\bf 41}, no.~1 (2001), 123--125.


\bibitem[BS99]{BS99}
{\sc A. B\"ottcher, B. Silbermann},
{\em Introduction to large truncated Toeplitz matrices},
Universitext. Springer-Verlag, New York, 1999.

\bibitem[BS06]{BS06}
{\sc A. B\"ottcher, B. Silbermann},
{\em Analysis of Toeplitz operators}, 2nd edition. Prepared jointly with A.~Karlovich,
Springer Monographs in Mathematics, Springer, Berlin, 2006.

\bibitem[GIKMV]{GIKMV} {\sc L. Brightmore, G.P. Gehér, A.R. Its, V.E. Korepin, F. Mezzadri, M.Y. Mo, J.A. Virtanen},
Entanglement entropy of two disjoint intervals separated by one spin in a chain of free fermion,
{\em J. Phys. A} {\bf 53} (2020), no. 34, 31pp.


\bibitem[CFS]{CFS} 
{\sc  J.B. Conrey,  P.J. Forrester,  N.C. Snaith}, 
Averages of ratios of characteristic polynomials for the compact classical groups. 
{\em Int. Math. Res. Not.}  (2005), no. 7, 397--431.



\bibitem[CG]{CG}
{\sc K.F. Clancey, I. Gohberg,} {\em Factorization of matrix functions and singular integral
operators,} Oper. Theory: Adv. Appl., Vol.~3, Birkh\"auser, Basel, 1981.


\bibitem[CS07]{CS07}
{\sc J.B. Conrey,   N.C. Snaith},
Applications of the L-functions ratios conjectures.
 {\em Proc. Lond. Math. Soc. (3)} {\bf 94} (2007), no. 3, 594--646.


\bibitem[E03a]{E03a}
{\sc T. Ehrhardt}, 
A generalization of Pincus' formula and Toeplitz operator determinants. Arch. Math. (Basel) {\bf 80} (2003), no. 3, 302--309.

\bibitem[E03b]{E03b}
{\sc T. Ehrhardt}, 
 A new algebraic proof of the Szeg\"o-Widom limit theorem,
{\em Acta Math. Hungar.} {\bf 99(3)}, 233--261 (2003).

\bibitem[E04a]{E04a}
{\sc T. Ehrhardt},
Invertibility theory for Toeplitz plus Hankel operators and singular integral operators with flip, 
{\em J. Funct. Anal.} {\bf 208} (2004), 64--106.

\bibitem[E04b]{E04b}
{\sc T. Ehrhardt},
{\em Factorization theory for Toeplitz + Hankel operators and singular integral operator with flip,}
Habilitation theses, Technische Universit\"at Chemnitz, 2004, ii+253 pages. \\
({\tt   https://nbn-resolving.org/urn:nbn:de:swb:ch1-200401246}).



\bibitem[GC]{GC}
{\sc J.F. Geronimo, K.M. Case},  Scattering theory and polynomials orthogonal on the unit circle,
{\em  J. Math. Phys.} {\bf 20} (1979), 299--310.


\bibitem[GK]{GK} 
{\sc I.C. Gohberg, M.G.  Kre\u{i}n}, 
{\em Introduction to the theory of linear nonselfadjoint operators}, Translations of Mathematical Monographs, Vol.~18, American Mathematical Society, Providence, 1969.

\bibitem[IJK]{IJK}
{\sc A.R. Its, B. Q. Jin, V.E. Korepin},
 Entanglement in the XY spin chain. 
 {\em J. Phys. A} {\bf 38} (2005), 2975--2990.



\bibitem[IMM]{IMM}
{\sc A.R. Its, F. Mezzadri, M.Y. Mo},
Entanglement entropy in quantum spin chains with finite range interaction. 
{\em Comm. Math. Phys.} {\bf 284} (2008), no.~1, 117--185. 


\bibitem[JK]{JK}  
{\sc B.-Q. Jin, V.E. Korepin}, Quantum spin chain, Toeplitz determinants and the Fisher-Hartwig conjecture, 
{\em J. Statist. Phys.} \textbf{116} (2004), no. 1-4, 79--95.
 
\bibitem[LS]{LS}
G.S. Litvinchuk, I.M. Spitkovsky, 
{\em Factorization of measurable matrix functions, }
 Oper. Theory: Adv. Appl., Vol.~25,  Birkh\"auser, Basel, 1987.



\bibitem[Wi74]{Wi74} 
{\sc H. Widom},
Asymptotic behavior of block Toeplitz matrices and determinants.
{\em Advances in Math.} {\bf 13} (1974), 284--322. 


\bibitem[Wi75]{Wi75}
{\sc H. Widom}
On the limit of block Toeplitz determinants
{\em Proc. Amer. Math. Soc.} {\bf 50} (1975), 167--173.

\bibitem[Wi76]{Wi76}
{\sc H. Widom},
Asymptotic behavior of block Toeplitz matrices and determinants. II.
{\em Advances in Math.}{\bf  21} (1976), no.~1, 1--29. 


\end{thebibliography}
 \end{document}